\documentclass[11pt]{article}

\usepackage{bbm}
\usepackage{eqnarray}
\usepackage{color}
\usepackage{bm}
\usepackage{amssymb}
\usepackage{float}
\usepackage{subfigure}
\usepackage{amsfonts,amsmath,amsthm,fancyhdr,multirow,hyperref}
\usepackage{booktabs,longtable,authblk}
\usepackage{mathrsfs,hhline}

\usepackage[top=2.5cm, bottom=2.5cm, left=3cm, right=3cm]{geometry}
\newtheorem{theorem}{Theorem}[section]
\newtheorem{proposition}[theorem]{Proposition}
\newtheorem{lemma}[theorem]{Lemma}
\newtheorem{corollary}[theorem]{Corollary}
\theoremstyle{definition}
\newtheorem{problem}[theorem]{Problem}
\newtheorem{definition}[theorem]{Definition}
\theoremstyle{remark}
\newtheorem{remark}[theorem]{Remark}

\newcommand{\R}{\mathbb R}
\newcommand{\C}{\mathbb C}
\newcommand{\N}{\mathbb N}
\newcommand{\B}{\mathcal B}
\newcommand{\norm}[1]{\lVert #1\rVert}
\newcommand{\snorm}[1]{\lVert #1\rVert_{[-1,1]}}
\newcommand{\Ki}{K_{\mathrm i}}
\newcommand{\sgnset}{\{-1,1\}}
\DeclareMathOperator{\sinc}{sinc}
\numberwithin{equation}{section}

\title{Erdős's Robust Polynomial Interpolation Problem: the Optimal Exponential Scale$^\dag$\footnotetext{\dag~Email addresses: jqyang24@m.fudan.edu.cn (J.-Q. Yang)
}
}
\author{Jia-Qi Yang}
\affil{School of Mathematical Sciences and Shanghai Key Laboratory for
	Contemporary Applied Mathematics, Fudan University, Shanghai 200433, China.}

\date{}

\begin{document}
\maketitle
\begin{abstract}
We determine the optimal exponential scale in Erd\H{o}s's robust
polynomial interpolation problem and identify the leading exponential
coefficient for Chebyshev--Lobatto grids. Given $C>0$ and
$0\le\rho<1$, put $H=C/(1-\rho)$. We prove that every configuration
of $n\ge n_0$ nodes in $[-1,1]$ admits sign data for which every
polynomial of degree at most $(1+\epsilon)n$ and uniform norm at most
$C$ has error greater than $\rho$ at more than $\epsilon n$ nodes.
One may choose $\epsilon\ge e^{-A(1+H)}$ and
$n_0\le e^{A(1+H)}$, where $A$ is an absolute constant, and a
matching upper bound shows that the exponential order is optimal.
For the same obstruction restricted to Chebyshev--Lobatto grids,
the optimal parameter $\epsilon_{\mathrm{Ch}}(C,\rho)$ satisfies
$\log(1/\epsilon_{\mathrm{Ch}}(C,\rho))
=\frac{\pi}{2}H+O(\log(H+1))$ as $H\to\infty$.
The proofs combine finite interpolation estimates in Bernstein spaces
with periodic sign constructions and averaging.
We also establish obstructions for random signs and stability under
angular perturbations of the nodes.
\end{abstract}

{\textbf{keywords:} Polynomial interpolation, Bernstein spaces,
interpolation constants, Chebyshev--Lobatto nodes,
Erd\H{o}s problem}

\section{Introduction}\label{sec:introduction}

For a polynomial $P$, write
\[
 \snorm{P}=\max_{-1\le t\le1}|P(t)|.
\]
Erd\H{o}s proposed the following assertion concerning interpolation with
both omitted data and a surplus in the polynomial degree
\cite[p.~72]{Erdos1968}. It is recorded as Erd\H{o}s Problem~1133.

\begin{problem}[Erd\H{o}s's robust interpolation problem]\label{prob:erdos}
For every $C>0$, there exist $\epsilon\in(0,1)$ and $n_0\in\N$ such
that, for every integer $n\ge n_0$ and every choice of distinct nodes
\[
 -1\le x_1<\cdots<x_n\le1,
\]
there is a vector $y=(y_1,\ldots,y_n)\in[-1,1]^n$ for which the
following implication holds for every $P\in\R[X]$:
\begin{equation}\label{eq:erdosproblem}
 \left.
 \begin{gathered}
  \deg P<(1+\epsilon)n,\\
  \#\{i\in\{1,\ldots,n\}:P(x_i)=y_i\}\ge(1-\epsilon)n
 \end{gathered}
 \right\}
 \quad\Longrightarrow\quad \snorm{P}>C.
\end{equation}
\end{problem}

The robustness lies in the simultaneous allowance of extra degree and
omitted data. The interpolating polynomial is not unique, and its set
of fitted nodes may vary. Thus a lower bound for a fixed Lagrange
interpolant does not by itself establish \eqref{eq:erdosproblem}.

A qualitative proof of this assertion is given in the unsigned draft
manuscript \cite{Ulam2026}. It passes to angular coordinates, obtains a
finite obstruction from Beurling's strict interpolation-density theorem,
and places these obstructions on disjoint groups of nodes. Its passage
from finite configurations to an infinite interpolation set uses
translation averaging and stationary point configurations. The argument
does not give a quantitative dependence of the obstruction parameter
on $C$ \cite[Section~6.3]{Ulam2026}.

Our first result determines the exponential order of this parameter,
including its dependence on a prescribed approximation tolerance
$0\le\rho<1$. With $H=C/(1-\rho)$, an obstruction is possible with
$\epsilon\ge e^{-A(1+H)}$ once $n\ge n_0$, where
$n_0\le e^{A(1+H)}$. Conversely, interpolation at Chebyshev--Lobatto
nodes rules out parameters larger than $e^{-\pi(H-3)/2}$ when $H>3$.
Thus the logarithm of the reciprocal of the optimal parameter has
order $H$.

Our second result identifies the leading exponential coefficient for
the full Chebyshev--Lobatto grids. If $\epsilon_{\mathrm{Ch}}(C,\rho)$
denotes the optimal parameter for this prescribed family, then
\[
 \log\frac1{\epsilon_{\mathrm{Ch}}(C,\rho)}
 =\frac\pi2 H+O(\log(H+1)).
\]
For the lower bound, we construct a periodic sign sequence from a
two-point evaluation functional on trigonometric polynomials. Averaging
over its periods controls arbitrary exceptional samples and reduces
the obstruction to a finite interpolation identity. A localized
cardinal formula gives the matching upper bound. This sharp coefficient
is proved for the Chebyshev grids; the corresponding coefficient for
arbitrary node configurations is not determined here.

The quantitative input for arbitrary nodes is the finite
Bernstein-space estimate of Olevskii and Ulanovskii
\cite[Proposition~4.2]{OU2018}. Their work already establishes
logarithmic growth of interpolation constants near critical density and
an exponential density gap for interpolation sets with bounded
interpolation constant \cite[Theorem~1(ii) and Corollary~5.1(ii)]{OU2018}.
We apply the finite estimate to bound the size of a local obstruction;
convex separation then produces sign data at the prescribed tolerance.
The angular reduction and grouping principle are those of
\cite{Ulam2026}. The quantitative bounds, their comparison with
Chebyshev interpolation, and the sharp coefficient on the full grids
are the main conclusions of this paper. The sign and tolerance
refinements already follow qualitatively from the earlier finite
obstruction.

Section~\ref{sec:mainresults} states the main results.
Sections~\ref{sec:bernstein} and~\ref{sec:counting} prove the
arbitrary-node bound. Section~\ref{sec:upper} gives the Chebyshev upper
construction, and Section~\ref{sec:sharp} proves the matching leading
coefficient. Section~\ref{sec:consequences} treats random signs and
perturbations, and Section~\ref{sec:remarks} records further questions.
The kernel estimates and an alternative qualitative proof by
localization are placed in Appendices~\ref{sec:cardinalappendix}
and~\ref{sec:appendix}.

\section{Main results}\label{sec:mainresults}

\subsection{Notation and conventions}

We take $\N=\{1,2,\ldots\}$ and write $[n]=\{1,\ldots,n\}$.
All logarithms are natural. A node configuration is an indexed vector
$x=(x_1,\ldots,x_n)\in[-1,1]^n$; unlike the distinct-node formulation
of Problem~\ref{prob:erdos}, our results allow repeated nodes. Counts
always refer to indices, including when some node locations coincide.
Unless stated otherwise, polynomials have complex coefficients.

For $x\in[-1,1]^n$, $y\in[-1,1]^n$, and $\rho\ge0$, define
\begin{equation}\label{eq:errordef}
 E_\rho(P;x,y)
 :=\#\{i\in[n]:|P(x_i)-y_i|>\rho\}.
\end{equation}
Thus an error equal to $\rho$ is not counted. For $D\ge0$ and $C>0$,
put
\begin{equation}\label{eq:polynomialclass}
 \mathcal P_D(C)
 :=\{P\in\C[X]:\deg P\le D,\ \snorm{P}\le C\}.
\end{equation}
For a nonintegral $D$, the degree condition means
$\deg P\le\lfloor D\rfloor$. We use $\deg0=-\infty$, so the zero
polynomial belongs to every class in \eqref{eq:polynomialclass}.
A constant is called absolute if it is independent of all nodes and
all displayed parameters.

\subsection{A quantitative robust obstruction}

\begin{theorem}[Quantitative robust obstruction]\label{thm:main}
There is an absolute constant $A>0$ such that, for every $C>0$ and
$0\le\rho<1$, one can choose $\epsilon\in(0,1)$ and $n_0\in\N$
satisfying
\begin{equation}\label{eq:mainparameters}
 \epsilon\ge\exp\!\left[-A\left(1+\frac{C}{1-\rho}\right)\right],
 \qquad
 n_0\le\exp\!\left[A\left(1+\frac{C}{1-\rho}\right)\right],
\end{equation}
such that, for every $n\ge n_0$ and every $x\in[-1,1]^n$, there exists
$y\in\sgnset^n$ with
\begin{equation}\label{eq:mainobstruction}
 E_\rho(P;x,y)>\epsilon n
 \qquad\text{for every }P\in\mathcal P_{(1+\epsilon)n}(C).
\end{equation}
\end{theorem}

Equivalently, these data satisfy
\begin{equation}\label{eq:mainimplication}
 \left.
 \begin{gathered}
  \deg P\le(1+\epsilon)n,\\
  \#\{i\in[n]:|P(x_i)-y_i|\le\rho\}\ge(1-\epsilon)n
 \end{gathered}
 \right\}
 \quad\Longrightarrow\quad\snorm{P}>C.
\end{equation}
At $\rho=0$, this implies Problem~\ref{prob:erdos}, with sign data,
complex polynomials, and possibly repeated nodes. The non-strict degree
bound in \eqref{eq:mainimplication} is also slightly stronger than the
strict degree bound in \eqref{eq:erdosproblem}.

\subsection{The optimal obstruction parameter}

\begin{definition}\label{def:optimal}
Fix $C>0$ and $0\le\rho<1$. A number $\epsilon\in(0,1)$ is
\emph{admissible} if the following assertion holds:
\begin{equation}\label{eq:admissible}
 \begin{aligned}
 &\exists n_0\in\N\quad\forall n\in\N,\ n\ge n_0
       \quad\forall x\in[-1,1]^n\\
 &\qquad\exists y\in\sgnset^n\quad
       \forall P\in\mathcal P_{(1+\epsilon)n}(C),
       \qquad E_\rho(P;x,y)>\epsilon n.
 \end{aligned}
\end{equation}
Define
\begin{equation}\label{eq:optimaldef}
 \epsilon_*(C,\rho)
 :=\sup\{\epsilon\in(0,1):\epsilon\text{ is admissible}\}.
\end{equation}
No quantitative restriction on $n_0$ is imposed in this definition.
\end{definition}

Theorem~\ref{thm:main} shows that the set in
\eqref{eq:optimaldef} is nonempty. It is downward closed: if $\epsilon$
is admissible and $0<\epsilon'<\epsilon$, the same sign data work for
$\epsilon'$, since the polynomial class and the required error count
both decrease. No assertion that the supremum is attained is needed.

\begin{theorem}[Optimal exponential order]\label{thm:order}
There is an absolute constant $A>0$ such that, for every $C>0$ and
$0\le\rho<1$, with $H=C/(1-\rho)$, one has
\begin{equation}\label{eq:order}
 e^{-A(1+H)}\le\epsilon_*(C,\rho).
\end{equation}
If $H>3$, then
\begin{equation}\label{eq:orderupper}
 \epsilon_*(C,\rho)\le\exp\!\left[\frac\pi2(3-H)\right].
\end{equation}
The upper bound also holds for the supremum obtained by allowing
$y\in[-1,1]^n$ in \eqref{eq:admissible}, instead of requiring sign data.
\end{theorem}

In particular,
\begin{equation}\label{eq:exponentialorder}
 \log\frac1{\epsilon_*(C,\rho)}\asymp\frac{C}{1-\rho}
 \qquad\text{as }\frac{C}{1-\rho}\longrightarrow\infty,
\end{equation}
where the comparison constants are absolute. This determines the
exponential order, not the leading coefficient in the exponent.

\subsection{The sharp exponential constant on Chebyshev grids}

For $n\ge2$, let
\begin{equation}\label{eq:fullgrid}
 x^{(n)}=(x_1^{(n)},\ldots,x_n^{(n)}),\qquad
 x_i^{(n)}=\cos\frac{\pi(i-1)}{n-1}.
\end{equation}
Define $\epsilon_{\mathrm{Ch}}(C,\rho)$ by
Definition~\ref{def:optimal}, with the quantifier over all
$x\in[-1,1]^n$ replaced by the single configuration $x^{(n)}$.
Thus the signs may depend on $n$, and they must force more than
$\epsilon n$ errors for every polynomial in
$\mathcal P_{(1+\epsilon)n}(C)$. As before, only the existence of a
threshold $n_0$ is required. Directly from the definitions,
\begin{equation}\label{eq:parametercomparison}
 \epsilon_*(C,\rho)\le\epsilon_{\mathrm{Ch}}(C,\rho).
\end{equation}

\begin{theorem}[Sharp exponential constant on Chebyshev grids]
\label{thm:sharp}
Let $C>0$, $0\le\rho<1$, and $H=C/(1-\rho)$. Then
\begin{equation}\label{eq:sharplower}
 \epsilon_{\mathrm{Ch}}(C,\rho)
 \ge\frac{e^{-\pi/2}}{4(H+1)}e^{-\pi H/2}.
\end{equation}
For $H>3$, one also has
\begin{equation}\label{eq:sharpupper}
 \epsilon_{\mathrm{Ch}}(C,\rho)\le e^{-\pi(H-3)/2}.
\end{equation}
Consequently,
\begin{equation}\label{eq:sharprate}
 \log\frac1{\epsilon_{\mathrm{Ch}}(C,\rho)}
 =\frac\pi2 H+O(\log(H+1))\qquad(H\to\infty),
\end{equation}
with an absolute implied constant. In particular, for each fixed
$0\le\rho<1$,
\begin{equation}\label{eq:sharplimit}
 \lim_{C\to\infty}\frac{1-\rho}{C}
       \log\frac1{\epsilon_{\mathrm{Ch}}(C,\rho)}=\frac\pi2.
\end{equation}
\end{theorem}

The estimate \eqref{eq:sharprate} is uniform as $H\to\infty$, including
when $\rho$ varies. It determines the leading exponential coefficient,
but not an asymptotic equivalent for $\epsilon_{\mathrm{Ch}}$.
The proof of \eqref{eq:sharplower} gives explicit periodic signs and an
explicit admissible parameter, but obtains the associated threshold
$n_0$ by compactness. It does not give the quantitative threshold in
Theorem~\ref{thm:main} for the parameter in \eqref{eq:sharplower}.
Furthermore, \eqref{eq:parametercomparison} does not transfer the new
lower bound to $\epsilon_*$.

\section{Finite obstructions in Bernstein spaces}\label{sec:bernstein}

\subsection{Bernstein spaces and interpolation constants}

For $\sigma>0$, let $\B_\sigma$ be the space of entire functions of
exponential type at most $\sigma$ that are bounded on $\R$, with norm
$\norm{f}_\infty=\sup_{t\in\R}|f(t)|$. Here exponential type at most
$\sigma$ means that, for every $\delta>0$, there is $M_\delta<\infty$
such that $|f(z)|\le M_\delta e^{(\sigma+\delta)|z|}$ for all $z\in\C$.
Let $\B_\sigma^{\R}$ denote the real subspace of functions real-valued
on $\R$. We use the standard vertical growth estimate
\begin{equation}\label{eq:growth}
 |f(x+iy)|\le\norm{f}_\infty e^{\sigma|y|}
 \qquad(x,y\in\R),
\end{equation}
and the completeness of $\B_\sigma$; see, for example,
\cite{OU2018,OrtegaCerdaSeip1999}. In particular, a series whose
Bernstein norms are summable converges in $\B_\sigma$ and locally
uniformly in $\C$.
The real symmetrization
\[
 f^{\R}(z)=\frac{f(z)+\overline{f(\bar z)}}2
\]
does not increase the norm or the exponential type, and does not increase
errors against real sample values. Cauchy's formula on a circle of radius
$1/\sigma$, together with \eqref{eq:growth}, gives the sufficient estimate
\begin{equation}\label{eq:derivative}
 \norm{f'}_\infty\le e\sigma\norm{f}_\infty.
\end{equation}

For a finite set $\Gamma\subset\R$, define its complex interpolation
constant by
\[
 \Ki(\Gamma,\B_\sigma)
 =\sup_{\norm{v}_{\ell^\infty(\Gamma)}\le1}
   \inf\{\norm{f}_\infty:f\in\B_\sigma,\ f|_\Gamma=v\}.
\]
The data $v$ are complex-valued, and the infimum is taken over all
interpolants, with $\inf\varnothing=+\infty$. In particular, no linear
choice of an interpolant as a function of the data is required.

\subsection{An effective finite obstruction}

The quantitative input is the following finite interpolation estimate.

\begin{theorem}[Olevskii--Ulanovskii {\cite[Proposition~4.2]{OU2018}}]
\label{thm:OU}
There are absolute constants $a,b>0$ such that, if $\tau\in\N$,
$\Gamma\subset[-\tau,\tau)$ is finite with $\#\Gamma\ge2\tau$, and
$\sigma>\pi$, then
\begin{equation}\label{eq:OU}
 \Ki(\Gamma,\B_\sigma)
 \ge a\log\frac{\sigma}{\sigma-\pi+b\tau^{-1/3}}.
\end{equation}
\end{theorem}

The constants are independent of the minimum distance between the
nodes. We use the estimate only when $\#\Gamma=2\tau$ and
$\pi<\sigma\le2\pi$.

\begin{lemma}[Effective finite obstruction]\label{lem:finite}
There is an absolute constant $A_0>0$ such that, for every $H>0$, one can
choose an even integer $k\ge2$ with
\begin{equation}\label{eq:kbound}
 k\le e^{A_0(1+H)}
\end{equation}
for which the following holds. For every indexed $k$-tuple in an interval
of length $k+1$, some real data of absolute value at most one have no
interpolant in $\B_\pi$ of norm at most $H$.
\end{lemma}

\begin{proof}
The integer $k$ will be chosen below. If two indexed nodes coincide,
assign the values $1$ and $-1$ at these two indices; then no function
can interpolate all the data. We may therefore assume that the nodes
are distinct. Translate the containing interval to $[0,k+1]$ and put
\[
 \lambda_k=\frac{k}{k+2},\qquad
 s_i=\lambda_k\left(t_i-\frac{k+1}{2}\right),\qquad
 \sigma_k=\frac{\pi}{\lambda_k}.
\]
Then $s_i\in(-k/2,k/2)$ and $\pi<\sigma_k\le2\pi$. If every real unit
data vector had a norm-$H$ interpolant at the $t_i$, the change of
variable
\[
 g(z)=f\left(\frac{z}{\lambda_k}+\frac{k+1}{2}\right)
\]
would give the same assertion in $\B_{\sigma_k}$ at the $s_i$. Splitting
complex data into their real and imaginary parts would imply
\begin{equation}\label{eq:complexconversion}
 \Ki(\{s_i\},\B_{\sigma_k})\le2H.
\end{equation}
On the other hand, Theorem~\ref{thm:OU}, with $\tau=k/2$, yields
\[
 \Ki(\{s_i\},\B_{\sigma_k})
 \ge a\log\frac{\pi}{B k^{-1/3}}
 =\frac a3\log k-b_0.
\]
Here we used $\sigma_k\ge\pi$ and
$\sigma_k-\pi+b(k/2)^{-1/3}\le(2\pi+b2^{1/3})k^{-1/3}$.
We have set $B=2\pi+b2^{1/3}$ and $b_0=a\log(B/\pi)>0$. Choose the smallest
even integer not less than
\[
 \exp\left(\frac{3(2H+b_0+1)}a\right).
\]
The last lower bound is at least $2H+1$, contradicting
\eqref{eq:complexconversion}. Rounding increases $k$ by less than two;
enlarging an absolute constant gives \eqref{eq:kbound}.
\end{proof}

The qualitative finite obstruction also follows directly
from Beurling's strict density theorem. We give that alternative proof in
Appendix~\ref{sec:appendix}. The estimate \eqref{eq:kbound} is the part
for which the quantitative theorem above is used.

\subsection{Sign data and prescribed tolerance}

\begin{lemma}[Signs and prescribed tolerance]\label{lem:signs}
Fix $C>0$, $0\le\rho<1$, and let $k$ be supplied by
Lemma~\ref{lem:finite} for $H=C/(1-\rho)$. For every indexed $k$-tuple
$t_1,\ldots,t_k$ in an interval of length $k+1$, there is
$s\in\sgnset^k$ such that
\begin{equation}\label{eq:localbad}
 \max_{1\le i\le k}|f(t_i)-s_i|>\rho
 \quad\text{for every }f\in\B_\pi\text{ with }\norm{f}_\infty\le C.
\end{equation}
The same assertion holds for $-s$.
\end{lemma}

\begin{proof}
Put $H=C/(1-\rho)$ and consider the set of real sample vectors
\[
 K=\{(g(t_1),\ldots,g(t_k)):
            g\in\B_\pi^{\R},\ \norm{g}_\infty\le H\}\subset\R^k.
\]
This set is convex and symmetric. It is compact by
\eqref{eq:growth} and the normal-family theorem: a locally uniform
limit preserves the type, the norm bound, and the sample values.
Lemma~\ref{lem:finite} supplies $v\in[-1,1]^k\setminus K$.
Strict separation gives a nonzero $w\in\R^k$ such that
\[
 \sup_{z\in K} w\cdot z<w\cdot v\le\sum_{i=1}^k|w_i|.
\]
Choose $s_i=\operatorname{sgn}(w_i)$ when $w_i\ne0$, and choose either
sign when $w_i=0$. If $f\in\B_\pi$ has norm at most $C$, the sample
vector of $f^{\R}/(1-\rho)$ lies in $K$. Hence
\[
 \sum_{i=1}^k w_i\operatorname{Re}f(t_i)
 <(1-\rho)\sum_{i=1}^k|w_i|.
\]
For at least one index with $w_i\ne0$, it follows that
$s_i\operatorname{Re}f(t_i)<1-\rho$, and therefore
$|f(t_i)-s_i|>\rho$. This proves \eqref{eq:localbad}.
Negating an approximant shows that the same assertion holds for $-s$.
\end{proof}

\section{From local obstructions to polynomial errors}\label{sec:counting}

\subsection{The finite counting inequality}

The angular substitution converts a degree bound into an exponential-type
bound. Disjoint groups of nodes then convert the local obstruction into
a lower bound for the total number of errors.

\begin{theorem}[Finite degree--error tradeoff]\label{thm:finite}
Fix $C>0$ and $0\le\rho<1$, and choose $k$ as in
Lemma~\ref{lem:signs}. For every $n\in\N$, $D>0$, and
$x\in[-1,1]^n$, there is $y\in\sgnset^n$ such that every polynomial
with $\deg P\le D$ and $\snorm{P}\le C$ satisfies
\begin{equation}\label{eq:finitecount}
 E_\rho(P;x,y)\ge
 \max\left\{0,
 \left\lceil\frac{n-(k-1)\lceil D/(k+1)\rceil}{k}\right\rceil
 \right\}.
\end{equation}
In particular,
\begin{equation}\label{eq:countlinear}
 E_\rho(P;x,y)\ge
 \frac{n}{k}-\frac{k-1}{k(k+1)}D-\frac{k-1}{k}.
\end{equation}
\end{theorem}

\begin{proof}
Put $t_i=(D/\pi)\arccos x_i\in[0,D]$. Partition this interval into
$N=\lceil D/(k+1)\rceil$ disjoint intervals of length at most $k+1$,
using half-open intervals except for the last right endpoint. If $n_j$
indexed nodes lie in the $j$th interval, split them into
$\lfloor n_j/k\rfloor$ disjoint groups of size $k$, leaving at most
$k-1$ indices unused. The total number $q$ of groups satisfies
\begin{equation}\label{eq:q}
 q=\sum_{j=1}^N\left\lfloor\frac{n_j}{k}\right\rfloor,
 \qquad kq\ge n-(k-1)N.
\end{equation}
On each group choose the signs from Lemma~\ref{lem:signs}, and assign
arbitrary signs to the unused indices.

If $P$ is an admissible polynomial, then
\[
 f(z)=P\left(\cos\frac{\pi z}{D}\right)
\]
is an entire trigonometric polynomial of exponential type at most
$\pi\deg P/D\le\pi$ (with the zero polynomial treated separately).
Since $\cos(\pi t/D)$ ranges over $[-1,1]$ for real $t$,
$\norm{f}_\infty=\snorm{P}\le C$. Moreover, $f(t_i)=P(x_i)$. Every complete group therefore contains at least one index
with error greater than $\rho$. The groups are disjoint, so
$E_\rho(P;x,y)\ge q$. This proves \eqref{eq:finitecount}; using
$N\le D/(k+1)+1$ proves \eqref{eq:countlinear}.
\end{proof}

\subsection{Degree surplus and error proportion}

\begin{corollary}[Separate degree and error proportions]\label{cor:tradeoff}
With $C,\rho,k$ as above, let
\[
 0\le\alpha<\frac2{k-1},\qquad
 a_k(\alpha)=\frac{2-(k-1)\alpha}{k(k+1)}.
\]
For every $0<\beta<a_k(\alpha)$ and all sufficiently large $n$, every
node configuration admits signs forcing
$E_\rho(P;x,y)>\beta n$ for all $P$ with
$\deg P\le(1+\alpha)n$ and $\snorm{P}\le C$.
It suffices that
\[
 n>\frac{k-1}{k(a_k(\alpha)-\beta)}.
\]
\end{corollary}

\begin{proof}
Substitute $D=(1+\alpha)n$ in \eqref{eq:countlinear}.
\end{proof}

\begin{proof}[Proof of Theorem~\ref{thm:main}]
Choose the even integer $k$ from Lemma~\ref{lem:signs}, and set
\begin{equation}\label{eq:explicitmain}
 \epsilon=\frac1{8k^2},\qquad n_0=4k^2.
\end{equation}
For $D=(1+\epsilon)n$, the grouping in Theorem~\ref{thm:finite}
satisfies, by \eqref{eq:q},
\[
 q\ge\frac nk-\frac{k-1}{k}
       \left\lceil\frac{D}{k+1}\right\rceil
 \ge\frac nk-\frac{D}{k+1}-1.
\]
The last, slightly weaker estimate suffices. Consequently,
\begin{align*}
 q-\epsilon n
 &\ge n\left(\frac{1-k\epsilon}{k(k+1)}-\epsilon\right)-1\\
 &=\frac{n}{k^2}\frac{7k-2}{8(k+1)}-1
 \ge\frac{n}{2k^2}-1\ge1,
\end{align*}
where $k\ge2$ and $n\ge4k^2$ were used. Thus $E_\rho\ge q>\epsilon n$.
Finally, \eqref{eq:kbound} and \eqref{eq:explicitmain} give
\eqref{eq:mainparameters} after increasing one absolute constant.
\end{proof}

\begin{remark}[Other bandlimited families]
The proof before its last change of variables is a statement about
$\B_\pi$ on an interval of length $D$. If nodes instead lie in an
interval of length $R$ and $f\in\B_\sigma$, rescaling replaces $D$ by
$\sigma R/\pi$ in \eqref{eq:finitecount}. Thus the same count applies
directly to entire functions of bounded exponential type and, in
particular, to trigonometric polynomials on a real interval.
\end{remark}

\section{Chebyshev interpolation and the upper bound}\label{sec:upper}

We prove the upper bounds in Theorems~\ref{thm:order}
and~\ref{thm:sharp} by constructing bounded interpolants on the
full Chebyshev--Lobatto grids. Filtered
interpolation at Chebyshev nodes provides related bounded interpolation
operators; see \cite{OccorsioThemistoclakis2021}. Here we use a direct
cardinal construction on the Chebyshev--Lobatto grid, keeping track of
the logarithmic dependence on the degree surplus.

\subsection{A localized cardinal construction}

\begin{proposition}[Interpolation with a small degree surplus]\label{prop:chebyshev}
Let $m\ge1$ and $0\le r\le m-1$ be integers. At the $m+1$ nodes
\[
 x_j=\cos\frac{\pi j}{m},\qquad 0\le j\le m,
\]
every complex data vector with $|y_j|\le1$ has an algebraic polynomial
interpolant $P$ satisfying
\begin{equation}\label{eq:chebbound}
 \deg P\le m+r,
 \qquad
 \snorm{P}\le3+\frac2\pi\log\frac{m}{r+1}.
\end{equation}
\end{proposition}

The proof is given in Appendix~\ref{sec:cardinalappendix}. It multiplies
the trigonometric cardinal kernel by a Fej\'er kernel of degree $r$,
normalized to equal one at zero. This preserves interpolation and gives the logarithmic
bound with coefficient $2/\pi$ in \eqref{eq:chebbound}.

\subsection{Approximation at all nodes and the optimal parameter}

\begin{corollary}[An upper bound for any degree allowance]\label{cor:upper}
Fix $0<\alpha<1$ and $0\le\rho<1$. For every integer $N\ge2$, every data vector $y\in[-1,1]^N$
on the full grid $x^{(N)}$ in \eqref{eq:fullgrid}
admits a polynomial $Q$ satisfying
\begin{equation}\label{eq:approxupper}
 \deg Q<(1+\alpha)N,\qquad
 \max_i|Q(x_i^{(N)})-y_i|\le\rho,
 \qquad
 \snorm{Q}\le(1-\rho)\left(3+\frac2\pi\log\frac1\alpha\right).
\end{equation}
\end{corollary}

\begin{proof}
Set $m=N-1$ and $r=\lfloor\alpha m\rfloor$, so $0\le r\le m-1$.
Proposition~\ref{prop:chebyshev} gives an exact interpolant $P$ at the
Chebyshev nodes. Since $r+1>\alpha m$,
\[
 \snorm{P}\le3+\frac2\pi\log\frac1\alpha,
 \qquad \deg P\le(1+\alpha)m<(1+\alpha)N.
\]
Take $Q=(1-\rho)P$. At each node its error is
$\rho|y_i|\le\rho$, and its norm has the bound in
\eqref{eq:approxupper}.
\end{proof}

\begin{proof}[Proof of Theorem~\ref{thm:order} and \eqref{eq:sharpupper}]
The lower bound is Theorem~\ref{thm:main}. If $H>3$ and an allowance
$\epsilon\in(0,1)$ satisfies
$\epsilon\ge\exp[(\pi/2)(3-H)]$, then
\[
 3+\frac2\pi\log(1/\epsilon)\le H.
\]
Corollary~\ref{cor:upper}, with $\alpha=\epsilon$, supplies for every
$N\ge2$ the full Chebyshev grid on which every bounded real data vector
has a polynomial of the required degree and norm with no errors
greater than $\rho$. Thus $\epsilon$ is inadmissible even for
$\epsilon_{\mathrm{Ch}}$. Taking the supremum proves
\eqref{eq:sharpupper}; \eqref{eq:parametercomparison} then gives
\eqref{eq:orderupper}.
\end{proof}

\section{The sharp constant on Chebyshev grids}\label{sec:sharp}

We prove the lower bound in Theorem~\ref{thm:sharp}. The argument uses
a periodic sign sequence whose sample values force a large two-point
evaluation functional. Averaging over periods then bounds the loss
caused by all exceptional samples together.

\subsection{A cyclic interpolation functional}

Let $k\ge2$ be even, and let $\mathcal T_k$ be the real trigonometric
polynomials of period $2k$ with frequencies $\ell\pi/k$, $|\ell|\le k$.
Their norm is the supremum on $\R$.

\begin{lemma}\label{lem:cyclic}
For $0\le j<2k$, put
\begin{equation}\label{eq:cyclicweights}
 w_j=\frac{(-1)^j}{2k\sin(\pi(1/2-j)/k)}.
\end{equation}
Every $F\in\mathcal T_k$ satisfies
\begin{equation}\label{eq:cyclicfunctional}
 \sum_{j=0}^{2k-1}w_jF(j)
 =\frac{F(1/2)-F(k+1/2)}2,
 \qquad
 \left|\sum_{j=0}^{2k-1}w_jF(j)\right|\le\norm{F}_\infty.
\end{equation}
Moreover,
\begin{equation}\label{eq:weightbounds}
 W_k:=\sum_{j=0}^{2k-1}|w_j|\ge\frac2\pi\log(k+1),
 \qquad
 \norm{w}_{\ell^\infty}\le\frac12.
\end{equation}
\end{lemma}

\begin{proof}
The cardinal kernel
\[
 L_k(t)=\frac{1+2\sum_{\ell=1}^{k-1}\cos(\ell\pi t/k)+\cos(\pi t)}{2k}
       =\frac{\sin(\pi t)}{2k}\cot\frac{\pi t}{2k}
\]
belongs to $\mathcal T_k$ and satisfies
$L_k(j)=\mathbf1_{\{j=0\}}$ for $0\le j<2k$, with removable
singularities filled in. The sample map from the $(2k+1)$-dimensional
space $\mathcal T_k$ onto $\R^{2k}$ therefore has a one-dimensional
kernel, spanned by $\sin(\pi t)$. Thus
\begin{equation}\label{eq:cyclicrepresentation}
 F(t)=\sum_{j=0}^{2k-1}F(j)L_k(t-j)+b\sin(\pi t)
\end{equation}
for some $b\in\R$. Since $k$ is even, the difference of evaluations
at $1/2$ and $k+1/2$ annihilates the last term. Applying it to the
cardinal functions, and using
$\cot u-\cot(u+\pi/2)=2/\sin(2u)$, gives
\eqref{eq:cyclicweights} and~\eqref{eq:cyclicfunctional}.

The absolute weights satisfy
\[
 W_k=\frac1k\sum_{j=0}^{k-1}\csc\frac{\pi(j+1/2)}k.
\]
By symmetry and $\sin u\le u$,
\[
 W_k\ge\frac2\pi\sum_{j=0}^{k/2-1}\frac1{j+1/2}
 \ge\frac2\pi\int_0^{k/2}\frac{du}{u+1/2}
 =\frac2\pi\log(k+1).
\]
Finally, $\sin u\ge2u/\pi$ for $0\le u\le\pi/2$ gives
\[
 \norm{w}_{\ell^\infty}
 =\frac1{2k\sin(\pi/(2k))}\le\frac12.
\]
\end{proof}

\subsection{Period averaging and the exceptional samples}

The next proposition states the obstruction before the choice of
parameters. Its sign sequence is independent of the locations of the
exceptional samples.

\begin{proposition}\label{prop:periodicobstruction}
Let $C>0$, $0\le\rho<1$, and $H=C/(1-\rho)$. Suppose $k\ge2$ is even
and $\epsilon\in(0,1)$ satisfies
\begin{equation}\label{eq:periodiccriterion}
 \epsilon k<1,\qquad H<W_k-k\epsilon(H+1),
\end{equation}
where $W_k$ is defined in Lemma~\ref{lem:cyclic}.
Extend $s_j=\operatorname{sgn}(w_j)$ periodically from
$0\le j<2k$ to all integers. Then, for all sufficiently large $n$,
the signs $y_i=s_{i-1}$ on $x^{(n)}$ satisfy
\[
 E_\rho(P;x^{(n)},y)>\epsilon n
 \qquad\text{for every }P\in\mathcal P_{(1+\epsilon)n}(C).
\]
\end{proposition}

\begin{proof}
Suppose the assertion fails for arbitrarily large $n$. Choose, along
an unbounded sequence, polynomials $P_n$ of the stated degree and norm
with at most $\epsilon n$ errors. Taking real parts does not increase
the norm or the errors against real data, so we may assume $P_n$ is
real. Put $m=n-1$ and
\[
 f_n(t)=P_n\!\left(\cos\frac{\pi t}{m}\right),\qquad
 \sigma_n=\pi(1+\epsilon)\frac{n}{n-1}.
\]
Then $f_n\in\B_{\sigma_n}^{\R}$ and $\norm{f_n}_\infty\le C$.
For all large $n$, let $M_n=\lfloor n/(2k)\rfloor\ge1$ and define
\begin{equation}\label{eq:periodaverage}
 g_n(t)=\frac1{M_n}\sum_{\ell=0}^{M_n-1}f_n(t+2k\ell).
\end{equation}
The vertical growth estimate gives
\[
 |g_n(z)|\le C e^{\sigma_n|\operatorname{Im}z|}.
\]
By the normal-family theorem, a subsequence converges locally uniformly
to an entire function $g$, real on $\R$, with
\begin{equation}\label{eq:limitgrowth}
 |g(z)|\le C e^{\pi(1+\epsilon)|\operatorname{Im}z|}.
\end{equation}
The telescoping identity
\[
 g_n(t+2k)-g_n(t)
 =\frac{f_n(t+2kM_n)-f_n(t)}{M_n}
\]
and the same growth estimate show that the left-hand side tends to
zero locally uniformly. Hence $g$ is $2k$-periodic.

We claim that $g\in\mathcal T_k$. For an integer $\ell$, let
\[
 \widehat g(\ell)=\frac1{2k}\int_0^{2k}
                 g(t)e^{-i\ell\pi t/k}\,dt.
\]
Periodicity and Cauchy's theorem allow the path of integration to be
shifted to height $y$. By \eqref{eq:limitgrowth},
\[
 |\widehat g(\ell)|
 \le C\exp\left(\pi(1+\epsilon)|y|+\frac{\ell\pi y}{k}\right).
\]
Letting $y\to-\infty$ for $\ell>k(1+\epsilon)$, or
$y\to+\infty$ for $\ell<-k(1+\epsilon)$, proves that these Fourier
coefficients vanish. Since $\epsilon k<1$, only the frequencies
$|\ell|\le k$ remain. Fourier uniqueness on the circle gives the claim.

For $0\le j<2k$, let $B_{n,j}$ count the erroneous samples among the
indices $i=j+2k\ell+1$, $0\le\ell<M_n$. These indices are at most
$2kM_n\le n$, and all their labels equal $s_j$. At a good sample,
$s_jf_n(j+2k\ell)\ge1-\rho$; at a bad sample, it is at least $-C$.
Consequently,
\begin{align}
 \sum_{j=0}^{2k-1}w_jg_n(j)
 &\ge(1-\rho)W_k-
       \frac{C+1-\rho}{M_n}\sum_{j=0}^{2k-1}|w_j|B_{n,j}
       \notag\\
 &\ge(1-\rho)W_k-
       \frac{C+1-\rho}{M_n}\norm{w}_{\ell^\infty}\epsilon n.
       \label{eq:weightederrors}
\end{align}
Here we used the total bound $\sum_jB_{n,j}\le\epsilon n$.
Passing to the limit, using $n/M_n\to2k$ and
Lemma~\ref{lem:cyclic} for $g$, yields
\[
 C\ge(1-\rho)W_k-
       2k\epsilon(C+1-\rho)\norm{w}_{\ell^\infty}.
\]
After division by $1-\rho$ and application of
\eqref{eq:weightbounds}, this becomes
\begin{equation}\label{eq:periodcontradiction}
 H\ge W_k-2k\epsilon(H+1)\norm{w}_{\ell^\infty}
   \ge W_k-k\epsilon(H+1),
\end{equation}
contrary to \eqref{eq:periodiccriterion}. This rules out every unbounded
sequence of failures and proves the assertion for all sufficiently
large $n$.
\end{proof}

\begin{proof}[Completion of the proof of Theorem~\ref{thm:sharp}]
Choose the smallest even integer $k$ such that
\[
 k\ge\exp\!\left[\frac\pi2(H+1)\right],\qquad
 \epsilon=\frac1{2k(H+1)}.
\]
Then $\epsilon k<1/2$ and
$W_k\ge(2/\pi)\log(k+1)>H+1$. Thus
$W_k-k\epsilon(H+1)>H+1/2$, so
Proposition~\ref{prop:periodicobstruction} makes $\epsilon$ admissible
for the Chebyshev grids. Since rounding increases $k$ by less than two
and $e^{\pi(H+1)/2}>2$,
\[
 k\le2e^{\pi(H+1)/2},\qquad
 \epsilon_{\mathrm{Ch}}(C,\rho)
 \ge\frac{e^{-\pi/2}}{4(H+1)}e^{-\pi H/2}.
\]
This proves \eqref{eq:sharplower}. Together with
\eqref{eq:sharpupper}, already proved in Section~\ref{sec:upper}, it
gives, for $H>3$,
\[
 \frac\pi2 H-\frac{3\pi}{2}
 \le\log\frac1{\epsilon_{\mathrm{Ch}}(C,\rho)}
 \le\frac\pi2 H+\log(H+1)+\frac\pi2+\log4.
\]
The remaining assertions follow.
\end{proof}

\begin{remark}
The averaging in \eqref{eq:periodaverage} is over a number of periods
that tends to infinity with $n$, while $k$ remains fixed. The weighted
estimate \eqref{eq:weightederrors} allows the exceptional samples to
be distributed arbitrarily among the periods. Thus the signs are chosen
before both the polynomial and its exceptional set, as required by the
definition of $\epsilon_{\mathrm{Ch}}$.
\end{remark}

\section{Further consequences}\label{sec:consequences}

\subsection{Random signs}\label{sec:random}

\begin{theorem}[Random-sign obstruction]\label{thm:randomintro}
For every $C>0$ and $0\le\rho<1$, there exist
$\alpha,\beta,\gamma>0$ and $n_1\in\N$ such that, for every $n\ge n_1$
and every fixed $x\in[-1,1]^n$, independent uniform signs
$Y_1,\ldots,Y_n$ satisfy
\begin{equation}\label{eq:randomintro}
 \mathbb P\!\left\{
 E_\rho(P;x,Y)>\beta n\ \text{for every }
 P\in\mathcal P_{(1+\alpha)n}(C)
 \right\}\ge1-e^{-\gamma n}.
\end{equation}
\end{theorem}

The constants depend only on $C$ and $\rho$, so the probability estimate
is uniform in the fixed node configuration. The quantifiers do not
assert that a single realization of $Y$ works for every configuration
simultaneously. Explicit choices of the constants are given in
\eqref{eq:randomconstants}. The degree surplus and the forced error
proportion may also be varied separately; see
Theorem~\ref{thm:finite} and Corollary~\ref{cor:tradeoff}.

For each complete group, Lemma~\ref{lem:signs} provides an antipodal
pair of sign patterns, either of which forces an error for every
admissible polynomial. The groups are disjoint, so the number of such
patterns that occur under random labeling is a sum of independent
Bernoulli variables. No discretization of the polynomial class is needed.

\begin{proposition}[A finite probability bound]\label{prop:random}
Use the groups in the proof of Theorem~\ref{thm:finite}, and let $q$ be
their number. Put $p=2^{1-k}$. For independent uniform signs
$Y\in\sgnset^n$, with probability at least $1-e^{-pq/8}$ every polynomial
with $\deg P\le D$ and $\snorm{P}\le C$ satisfies
\begin{equation}\label{eq:randomfinite}
 E_\rho(P;x,Y)\ge\frac{pq}{2}.
\end{equation}
\end{proposition}

\begin{proof}
Fix a bad vector $s_j$ in the $j$th complete group. Let $X_j$ indicate
that the random labels in this group equal $s_j$ or $-s_j$. The two
vectors are distinct, so $\mathbb P(X_j=1)=2/2^k=p$. The variables
$X_j$ are independent. For every realization and every admissible $P$,
Lemma~\ref{lem:signs} gives
\[
 E_\rho(P;x,Y)\ge X,\qquad X=\sum_{j=1}^q X_j.
\]
This inequality already holds simultaneously over all the polynomials.
For completeness, with $s=\log2$, the exponential Markov inequality gives
\begin{align*}
 \mathbb P\{X<pq/2\}
 &\le e^{spq/2}(1-p+pe^{-s})^q\\
 &\le\exp\bigl[pq(s/2+e^{-s}-1)\bigr]
 \le e^{-pq/8}.
\end{align*}
If $q=0$, the asserted lower probability bound is zero and the statement
is immediate. This proves the proposition in all cases.
\end{proof}

\begin{proof}[Proof of Theorem~\ref{thm:randomintro}]
For the same $k$, take
\begin{equation}\label{eq:randomconstants}
 \alpha=\frac1{k-1},\qquad
 \beta=\frac{2^{1-k}}{8k(k+1)},\qquad
 \gamma=\frac{2^{1-k}}{16k(k+1)},\qquad
 n_1=2(k^2-1).
\end{equation}
If $D=(1+\alpha)n$ and $n\ge n_1$, \eqref{eq:q} implies
\[
 q\ge\frac{n}{k(k+1)}-\frac{k-1}{k}
 \ge\frac{n}{2k(k+1)}.
\]
On the event in Proposition~\ref{prop:random},
$E_\rho\ge pq/2\ge2\beta n>\beta n$, and $pq/8\ge\gamma n$.
\end{proof}

In particular, at least $(1-e^{-\gamma n})2^n$ sign vectors satisfy
the conclusion. Writing $H=C/(1-\rho)$, the displayed choices satisfy
\[
 \alpha\ge e^{-A(1+H)},\qquad n_1\le e^{A(1+H)}
\]
for an absolute $A>0$. In contrast, the factors $2^{1-k}$ in $\beta$
and $\gamma$ may be doubly exponentially small in $H$. In particular,
a high probability in~\eqref{eq:randomintro} may require $n$ much larger
than the deterministic threshold. These probability parameters are not
asserted to be optimal.

\subsection{Stability under angular perturbations}\label{sec:stability}

The angular coordinate used above also gives stability on the natural
inverse-degree scale. Write
\[
 d_{\mathrm{ang}}(x,z)=|\arccos x-\arccos z|
 \qquad(x,z\in[-1,1]).
\]

\begin{proposition}[Uniform angular stability]\label{prop:stability}
Let $D>0$, $C>0$, and $0\le\rho<\rho'<1$. Suppose a sign vector $y$
at nodes $x$ has the property that every $P$ with
$\deg P\le D$, $\snorm{P}\le C$ has at least $q$ errors greater than
$\rho'$. Then the same vector has at least $q$ errors greater than
$\rho$ at every node array $z$ satisfying
\begin{equation}\label{eq:perturb}
 \max_i d_{\mathrm{ang}}(x_i,z_i)
 \le\frac{\rho'-\rho}{eCD}.
\end{equation}
The conclusion is simultaneous over all such arrays $z$ and all
admissible polynomials.
\end{proposition}

\begin{proof}
For a fixed admissible $P$, the entire function
$Q(\theta)=P(\cos\theta)$ belongs to $\B_D$ and has norm at most $C$.
By \eqref{eq:derivative},
\[
 |P(x_i)-P(z_i)|\le eCD\,d_{\mathrm{ang}}(x_i,z_i)
 \le\rho'-\rho.
\]
Each index at which $|P(x_i)-y_i|>\rho'$ therefore satisfies
$|P(z_i)-y_i|>\rho$. The estimate is uniform in both $P$ and $z$.
\end{proof}

\begin{corollary}[Random signs stable under perturbation]\label{cor:randomstable}
Fix $C>0$ and $0\le\rho<1$. There are
$\alpha,\beta,\gamma,h>0$ and $n_1$ such that, for each fixed
$x\in[-1,1]^n$ with $n\ge n_1$, the following holds with probability at
least $1-e^{-\gamma n}$ over uniform signs $Y$:
\[
 E_\rho(P;z,Y)>\beta n
\]
simultaneously for all $P$ and $z\in[-1,1]^n$ satisfying
\[
 \deg P\le(1+\alpha)n,\qquad \snorm{P}\le C,\qquad
 \max_i d_{\mathrm{ang}}(x_i,z_i)\le h/n.
\]
One may take the constants in \eqref{eq:randomconstants} for tolerance
$\rho'=(1+\rho)/2$, and
\[
 h=\frac{1-\rho}{2eC(1+\alpha)}.
\]
\end{corollary}

\begin{proof}
Apply Theorem~\ref{thm:randomintro} at tolerance $\rho'$ and then
Proposition~\ref{prop:stability}.
\end{proof}

\begin{remark}[Empirical error bounds]\label{rem:empirical}
If $0<\rho<1$, $\beta>0$, and $E_\rho(P;x,y)>\beta n$, then, for every $r>0$,
\[
 \left(\frac1n\sum_{i=1}^n|P(x_i)-y_i|^r\right)^{1/r}
 >\rho\,\beta^{1/r},
\]
since more than $\beta n$ summands exceed $\rho^r$.
Thus the deterministic and random-sign obstructions also give uniform
lower bounds for empirical $\ell^r$ errors.
\end{remark}

\section{Concluding remarks}\label{sec:remarks}

The elementary parameter ranges are easily described. If $H<1$, then
$1-C>\rho$, so every polynomial of norm at most $C$ has error greater
than $\rho$ at every sign sample. In this case,
$\epsilon_* = \epsilon_{\mathrm{Ch}}=1$. If $H\ge1$, the constants
$1-\rho$ and $-(1-\rho)$ have norm at most $C$; one of them fits at
least half of any sign vector to tolerance $\rho$. Hence
$\epsilon_*\le\epsilon_{\mathrm{Ch}}\le1/2$. Finally, a tolerance
$\rho\ge1$ would make the zero polynomial fit every sign sample, so
the restriction $\rho<1$ is necessary.

Theorem~\ref{thm:sharp} identifies the leading exponential coefficient
for the full Chebyshev--Lobatto grids. The corresponding question for
arbitrary nodes remains open: our estimates do not establish the
existence or value of a limit of
$H^{-1}\log(1/\epsilon_*(C,\rho))$ as $H\to\infty$.
Even on the Chebyshev grids, the present bounds leave a factor of
order $H+1$ between the upper and lower estimates, up to absolute
constants. They do not determine an asymptotic prefactor.

The optimal threshold in $n$ and the best relation between degree
surplus and forced error proportion are also unresolved. In particular,
the period-averaging argument in Section~\ref{sec:sharp} gives no
explicit threshold. For random signs, the proof counts only one
antipodal pair of forbidden patterns in each group. Bounds for a larger
collection of forbidden patterns could improve the forced error
fraction and the probability exponent.

\appendix
\section{The localized cardinal estimate}\label{sec:cardinalappendix}

This appendix proves Proposition~\ref{prop:chebyshev}.

\begin{proof}
Use the $2m$ equally spaced angles $\theta_j=\pi j/m$ and extend the
data evenly and periodically, so that $v_j=y_j$ for $0\le j\le m$ and
$v_{2m-j}=y_j$ for $1\le j<m$. Define the even trigonometric polynomials
\begin{align*}
 H_m(t)&=\frac{1+2\sum_{a=1}^{m-1}\cos(at)+\cos(mt)}{2m},\\
 W_r(t)&=\left(\frac{\sin((r+1)t/2)}{(r+1)\sin(t/2)}\right)^2,
 \qquad K_{m,r}(t)=H_m(t)W_r(t),
\end{align*}
with removable singularities filled in. Their degrees are $m$, $r$,
and at most $m+r$, respectively. The identity
\begin{equation}\label{eq:cardinalidentity}
 H_m(t)=\frac{\sin(mt)\cot(t/2)}{2m}
\end{equation}
away from the removable points shows that
$K_{m,r}(\theta_j)=\mathbf1_{\{j=0\}}$ for $0\le j<2m$.
Consequently,
\[
 T(t)=\sum_{j=0}^{2m-1}v_jK_{m,r}(t-\theta_j)
\]
interpolates the extended data. It is even, since both the kernel and the
data are even on the circle. Hence $T(t)=P(\cos t)$ for an algebraic
polynomial $P$ of degree at most $m+r$.

It remains to bound the norm. Let $d(t)$ be the distance from $t$ to
$2\pi\mathbb Z$, so $0\le d(t)\le\pi$. From the Fourier formula for
$H_m$, whose coefficients have total absolute value one, and from
\eqref{eq:cardinalidentity}, using $\cot(d/2)\le2/d$,
\[
 |H_m(t)|\le\min\left\{1,\frac{1}{m\,d(t)}\right\}.
\]
For real $t$, the geometric-sum identity gives
\[
 W_r(t)=\left|\frac1{r+1}\sum_{a=0}^{r}e^{iat}\right|^2\le1.
\]
Together with $\sin(d/2)\ge d/\pi$ for $0\le d\le\pi$, this yields
\[
 |W_r(t)|\le\min\left\{1,
       \frac{\pi^2}{(r+1)^2d(t)^2}\right\}.
\]
The bounds at $d=0$ are interpreted by continuity. Put $h=\pi/m$ and
$A=m/(r+1)\ge1$. For fixed $t$, each distance shell
$\ell h\le d(t-\theta_j)<(\ell+1)h$, $0\le\ell\le m-1$,
contains at most two grid points: each of its two half-open arcs has
length $h$, the grid spacing. A grid point at distance $\pi$, if one
exists, contributes nothing because $H_m(\pi)=0$. The
shell with $\ell=0$ contributes at most two. It follows that
\begin{equation}\label{eq:lebesguesum}
 \sum_{j=0}^{2m-1}|K_{m,r}(t-\theta_j)|
 \le2+\frac2\pi\sum_{\ell=1}^\infty\frac1\ell
                     \min\left\{1,\frac{A^2}{\ell^2}\right\}.
\end{equation}
The function $g(u)=u^{-1}\min\{1,A^2/u^2\}$ is decreasing on
$[1,\infty)$. Since $A\ge1$,
\[
 \sum_{\ell=1}^\infty g(\ell)
 \le g(1)+\int_1^\infty g(u)\,du
 =1+\log A+\frac12.
\]
Thus \eqref{eq:lebesguesum} is at most
$2+3/\pi+(2/\pi)\log A<3+(2/\pi)\log A$, proving
\eqref{eq:chebbound} because $|v_j|\le1$.
\end{proof}

\section{A qualitative finite obstruction by localization}\label{sec:appendix}

We give an alternative proof of the qualitative finite obstruction,
using Beurling's strict interpolation-density theorem in place of
Theorem~\ref{thm:OU}. Unlike the stationary extraction in
\cite{Ulam2026}, the argument constructs a single infinite interpolation
set by dilating and localizing finite blocks. No estimate for the block
size is obtained. Its parity is irrelevant to the grouping argument.

For a uniformly separated set $\Lambda\subset\R$, write
\[
 D^+(\Lambda)=\limsup_{R\to\infty}\sup_{a\in\R}
       \frac{\#(\Lambda\cap[a,a+R])}{R}.
\]
The necessary part of Beurling's interpolation theorem states that, if
every bounded complex data function on $\Lambda$ is the restriction of
some member of $\B_\pi$, then $D^+(\Lambda)<1$; see
\cite{Beurling1989,OrtegaCerdaSeip1999}. We will obtain a uniform
interpolation bound in the construction below, so uniform separation
will follow directly from \eqref{eq:derivative}.

We first record the successive-correction argument used below.

\begin{lemma}[Successive correction]\label{lem:correction}
Fix $\sigma>0$, and let $\Lambda=(\lambda_i)_{i\in I}$ be a finite or
infinite indexed family of real nodes. Write
$f|_\Lambda=(f(\lambda_i))_{i\in I}$.
Suppose $M>0$, $0\le\theta<1$, and every
$v\in\ell^\infty(I;\R)$ admits $f\in\B_\sigma^{\R}$ satisfying
\[
 \norm{f}_\infty\le M\norm{v}_\infty,
 \qquad \norm{v-f|_\Lambda}_\infty\le\theta\norm{v}_\infty.
\]
Then every such $v$ has an exact real interpolant of norm at most
$M\norm{v}_\infty/(1-\theta)$.
\end{lemma}

\begin{proof}
Set $r_0=v$. Successively choose $f_j$ for $r_j$ and put
$r_{j+1}=r_j-f_j|_\Lambda$. Then
\[
 \norm{r_j}_\infty\le\theta^j\norm{v}_\infty,
 \qquad
 \norm{f_j}_\infty\le M\theta^j\norm{v}_\infty.
\]
If a residual vanishes, take all subsequent terms to be zero. Completeness
and \eqref{eq:growth} show that $f=\sum_{j\ge0}f_j$ belongs to
$\B_\sigma^{\R}$, has the stated norm bound, and has samples equal to $v$.
\end{proof}

\begin{proposition}\label{prop:qualitative}
For every $H>0$, some integer $k\ge2$ has the following property:
every indexed $k$-tuple in an interval of length $k+1$ admits real unit
data with no norm-$H$ interpolant in $\B_\pi$.
\end{proposition}

\begin{proof}
For $H<1$, constant unit data already give the assertion. Assume $H\ge1$
and suppose the assertion is false. For every integer $j\ge2$, choose
a set
\[
 \Lambda_j\subset[-(j+1)/2,(j+1)/2],\qquad\#\Lambda_j=j,
\]
on which every real unit data vector has an interpolant in
$\B_\pi^{\R}$ of norm at most $H$. The nodes must be distinct, since
otherwise conflicting real data could not be interpolated.

\smallskip
\noindent\emph{Step 1: dilation and localization.}
Set
\[
 a_j=1+\frac1{(j+1)^2},\qquad
 \sigma_j=\frac\pi{a_j},\qquad
 b_j=\frac{\pi-\sigma_j}{2},\qquad
 \eta=\frac1{8H}.
\]
We will choose large positive translations $T_j$ and put
\[
 \Gamma_j=T_j+a_j\Lambda_j,\qquad
 w_j(z)=\sinc^2\bigl(b_j(z-T_j)\bigr),
 \qquad \sinc z=\frac{\sin z}{z},\quad\sinc0=1.
\]
For $t\in\Gamma_j$,
\[
 |b_j(t-T_j)|\le\frac\pi{4(j+1)}\le\frac\pi{12},
 \qquad w_j(t)\ge\frac12.
\]
The last inequality follows, for instance, from
$\sin u/u\ge1-u^2/6$ in this range.

Define, on the real line,
\[
 u_j(x)=\min\left\{1,\frac1{b_j^2|x-T_j|^2}\right\},
\]
where the value at $x=T_j$ is one. The integral identity
$\sinc z=\int_0^1\cos(tz)\,dt$ and the elementary bound
$|\sin z|\le e^{|\operatorname{Im}z|}$ give
\begin{equation}\label{eq:windowgrowth}
 |w_j(x+iy)|\le e^{2b_j|y|}u_j(x).
\end{equation}
Choose
\[
 L_j>\max\left\{\frac{a_j(j+1)}2,
                \frac1{b_j}\sqrt{\frac{2^j}{\eta}}\right\}.
\]
Choose $T_2=L_2+1$ and, successively,
$T_{j+1}=T_j+L_j+L_{j+1}+1$. Then the intervals
$J_j=[T_j-L_j,T_j+L_j]$ are pairwise disjoint, their left endpoints
tend to infinity, and $\Gamma_j\subset J_j$. For $x\notin J_j$,
$u_j(x)\le\eta2^{-j}$. Since a real $x$ lies in at most one $J_j$,
\begin{equation}\label{eq:sumwindows}
 \sum_{j\ge2}u_j(x)\le1+\eta\qquad(x\in\R).
\end{equation}

\smallskip
\noindent\emph{Step 2: gluing and successive correction.}
Let $v$ be any real unit data function on
$\Gamma=\bigcup_{j\ge2}\Gamma_j$. By dilation of the assumed
interpolation property on $\Lambda_j$, the data $v(t)/w_j(t)$ on
$\Gamma_j$ have an interpolant $g_j\in\B_{\sigma_j}^{\R}$ with
$\norm{g_j}_\infty\le2H$. Consider
\begin{equation}\label{eq:glued}
 F(z)=\sum_{j\ge2}w_j(z)g_j(z).
\end{equation}
Since $2b_j+\sigma_j=\pi$, each summand satisfies
\[
 |w_j(x+iy)g_j(x+iy)|\le2H e^{\pi|y|}u_j(x).
\]
For a fixed compact subset of $\C$, its real projection is disjoint
from $J_j$ for all sufficiently large $j$. The tail is therefore
uniformly dominated there by a constant times $\sum_j\eta2^{-j}$.
Thus \eqref{eq:glued} converges locally uniformly, and
\eqref{eq:sumwindows} gives
\[
 |F(x+iy)|\le2H(1+\eta)e^{\pi|y|}.
\]
In particular, $F\in\B_\pi^{\R}$ and
$\norm{F}_\infty\le2H(1+\eta)$.

At $t\in\Gamma_\ell$, its own summand equals $v(t)$, while all other
windows are small. Hence
\[
 |F(t)-v(t)|\le2H\sum_{j\ne\ell}u_j(t)
 \le2H\eta=\frac14.
\]
By scaling and Lemma~\ref{lem:correction}, every bounded real data
function on $\Gamma$ has an exact interpolant of norm at most
\[
 K\norm{v}_\infty,\qquad K=\frac{8H(1+\eta)}3.
\]
Splitting real and imaginary parts gives bounded complex interpolation
as well, with norm bound $2K\norm{v}_\infty$ for complex data.

\smallskip
\noindent\emph{Step 3: separation and the density contradiction.}
Given distinct $s,t\in\Gamma$, interpolate real data
equal to $1$ at $s$, $-1$ at $t$, and zero elsewhere. Equation
\eqref{eq:derivative} implies
\[
 2\le e\pi K|s-t|.
\]
Thus $\Gamma$ is uniformly separated and Beurling's theorem applies.

On the other hand, $\Gamma_j$ contains $j$ points in an interval of
length $R_j=a_j(j+1)$, and $R_j\to\infty$. Therefore
\[
 D^+(\Gamma)\ge\lim_{j\to\infty}\frac{j}{a_j(j+1)}=1,
\]
contradicting the strict inequality $D^+(\Gamma)<1$.
\end{proof}

Using Proposition~\ref{prop:qualitative} in place of
Lemma~\ref{lem:finite}, the separation and grouping arguments give
the qualitative obstruction and its consequences without
Theorem~\ref{thm:OU}. The latter theorem is needed for the
quantitative arbitrary-node bounds. The sharp result for Chebyshev
grids has the independent proof given in Section~\ref{sec:sharp}.

\end{document}